\documentclass[12pt]{amsart}

\usepackage[T1]{fontenc}        % Umlaute werden als eine Einheit angesehen -> richtige Trennung;
\usepackage{ae,aecompl}
\usepackage{enumerate}
\usepackage{latexsym}
\usepackage{amsmath,amsthm,amsfonts}
\usepackage{amssymb}
\usepackage{epic}
\usepackage{tikz}
\usetikzlibrary{calc}
\usepackage{graphicx}
\usepackage{hyperref} % für Druck entfernen?
\usepackage{url}
\usepackage{pgfplots}
\usepackage{subcaption}
\usepackage{algorithm}
\usepackage[noend]{algpseudocode}
\usepackage{float}
\usepackage{multicol}
\usetikzlibrary{matrix,shapes,decorations.pathreplacing}
\usepackage{easybmat}
\usepackage{multirow,bigdelim}

\makeatletter
\@namedef{subjclassname@2020}{%
  \textup{2020} Mathematics Subject Classification}
\makeatother

\usepackage[backend=biber]{biblatex} %biblatex mit biber laden
\DeclareFieldFormat{pages}{#1}
\title[Filling space with hypercubes of two sizes]{Filling space with hypercubes of two sizes --\\ The pythagorean tiling in higher dimensions}

\author{Jakob F\"uhrer}
\address{Institute of Analysis and Number Theory \\ Technical University of Graz\\
Kopernikusgasse 24/II\\
8010 Graz, Austria}
\email{jakob.fuehrer@tugraz.at}
\date{}

\usepackage{geometry}

\usepackage{tikz-3dplot}
\usetikzlibrary{calc}
\tikzset{plane/.style n args={3}{insert path={%
#1 -- ++ #2 -- ++ #3 -- ++ ($-1*#2$) -- cycle}},
unit xy plane/.style={plane={#1}{(1,0,0)}{(0,1,0)}},
unit xz plane/.style={plane={#1}{(1,0,0)}{(0,0,1)}},
unit yz plane/.style={plane={#1}{(0,1,0)}{(0,0,1)}},
get projections/.style={insert path={%
let \p1=(1,0,0),\p2=(0,1,0)  in 
[/utils/exec={\pgfmathtruncatemacro{\xproj}{sign(\x1)}\xdef\xproj{\xproj}
\pgfmathtruncatemacro{\yproj}{sign(\x2)}\xdef\yproj{\yproj}
\pgfmathtruncatemacro{\zproj}{sign(cos(\tdplotmaintheta))}\xdef\zproj{\zproj}}]}},
pics/unit cube/.style={code={
\path[get projections];
\draw(0,0,0) -- (1,1,1);
\ifnum\zproj=-1
 \path[3d cube/every face,3d cube/xy face,unit xy plane={(0,0,0)}]; 
\fi
\ifnum\yproj=1
 \path[3d cube/every face,3d cube/yz face,unit yz plane={(1,0,0)}]; 
\else
 \path[3d cube/every face,3d cube/yz face,unit yz plane={(0,0,0)}]; 
\fi
\ifnum\xproj=1
 \path[3d cube/every face,3d cube/xz face,unit xz plane={(0,0,0)}]; 
\else
 \path[3d cube/every face,3d cube/xz face,unit xz plane={(0,1,0)}]; 
\fi
\ifnum\zproj>-1
 \path[3d cube/every face,3d cube/xy face,unit xy plane={(0,0,1)}]; 
\fi
}},
pics/unit cube2/.style={code={
\path[get projections];
\draw (0,0,0) -- (1,1,1);
\ifnum\zproj=-1
 \path[3d cube/every face,3d cube/xy face2,unit xy plane={(0,0,0)}]; 
\fi
\ifnum\yproj=1
 \path[3d cube/every face,3d cube/yz face2,unit yz plane={(1,0,0)}]; 
\else
 \path[3d cube/every face,3d cube/yz face2,unit yz plane={(0,0,0)}]; 
\fi
\ifnum\xproj=1
 \path[3d cube/every face,3d cube/xz face2,unit xz plane={(0,0,0)}]; 
\else
 \path[3d cube/every face,3d cube/xz face2,unit xz plane={(0,1,0)}]; 
\fi
\ifnum\zproj>-1
 \path[3d cube/every face,3d cube/xy face2,unit xy plane={(0,0,1)}]; 
\fi
}},
pics/unit cube000/.style={code={
\path[get projections];
\draw(0,0,0) -- (1,1,1);
\ifnum\zproj=-1
 \path[thick,3d cube/every face,3d cube/xy face000,unit xy plane={(0,0,0)}]; 
\fi
\ifnum\yproj=1
 \path[thick,3d cube/every face,3d cube/yz face000,unit yz plane={(1,0,0)}]; 
\else
 \path[thick,3d cube/every face,3d cube/yz face000,unit yz plane={(0,0,0)}]; 
\fi
\ifnum\xproj=1
 \path[thick,3d cube/every face,3d cube/xz face000,unit xz plane={(0,0,0)}]; 
\else
 \path[thick,3d cube/every face,3d cube/xz face000,unit xz plane={(0,1,0)}]; 
\fi
\ifnum\zproj>-1
 \path[thick,3d cube/every face,3d cube/xy face000,unit xy plane={(0,0,1)}]; 
\fi
}},
pics/unit cube001/.style={code={
\path[get projections];
\draw(0,0,0) -- (1,1,1);
\ifnum\zproj=-1
 \path[3d cube/every face,3d cube/xy face001,unit xy plane={(0,0,0)}]; 
\fi
\ifnum\yproj=1
 \path[3d cube/every face,3d cube/yz face001,unit yz plane={(1,0,0)}]; 
\else
 \path[3d cube/every face,3d cube/yz face001,unit yz plane={(0,0,0)}]; 
\fi
\ifnum\xproj=1
 \path[3d cube/every face,3d cube/xz face001,unit xz plane={(0,0,0)}]; 
\else
 \path[3d cube/every face,3d cube/xz face001,unit xz plane={(0,1,0)}]; 
\fi
\ifnum\zproj>-1
 \path[3d cube/every face,3d cube/xy face001,unit xy plane={(0,0,1)}]; 
\fi
}},
pics/unit cube010/.style={code={
\path[get projections];
\draw(0,0,0) -- (1,1,1);
\ifnum\zproj=-1
 \path[3d cube/every face,3d cube/xy face010,unit xy plane={(0,0,0)}]; 
\fi
\ifnum\yproj=1
 \path[3d cube/every face,3d cube/yz face010,unit yz plane={(1,0,0)}]; 
\else
 \path[3d cube/every face,3d cube/yz face010,unit yz plane={(0,0,0)}]; 
\fi
\ifnum\xproj=1
 \path[3d cube/every face,3d cube/xz face010,unit xz plane={(0,0,0)}]; 
\else
 \path[3d cube/every face,3d cube/xz face010,unit xz plane={(0,1,0)}]; 
\fi
\ifnum\zproj>-1
 \path[3d cube/every face,3d cube/xy face010,unit xy plane={(0,0,1)}]; 
\fi
}},
pics/unit cube011/.style={code={
\path[get projections];
\draw(0,0,0) -- (1,1,1);
\ifnum\zproj=-1
 \path[3d cube/every face,3d cube/xy face011,unit xy plane={(0,0,0)}]; 
\fi
\ifnum\yproj=1
 \path[3d cube/every face,3d cube/yz face011,unit yz plane={(1,0,0)}]; 
\else
 \path[3d cube/every face,3d cube/yz face011,unit yz plane={(0,0,0)}]; 
\fi
\ifnum\xproj=1
 \path[3d cube/every face,3d cube/xz face011,unit xz plane={(0,0,0)}]; 
\else
 \path[3d cube/every face,3d cube/xz face011,unit xz plane={(0,1,0)}]; 
\fi
\ifnum\zproj>-1
 \path[3d cube/every face,3d cube/xy face011,unit xy plane={(0,0,1)}]; 
\fi
}},
pics/unit cube100/.style={code={
\path[get projections];
\draw(0,0,0) -- (1,1,1);
\ifnum\zproj=-1
 \path[3d cube/every face,3d cube/xy face100,unit xy plane={(0,0,0)}]; 
\fi
\ifnum\yproj=1
 \path[3d cube/every face,3d cube/yz face100,unit yz plane={(1,0,0)}]; 
\else
 \path[3d cube/every face,3d cube/yz face100,unit yz plane={(0,0,0)}]; 
\fi
\ifnum\xproj=1
 \path[3d cube/every face,3d cube/xz face100,unit xz plane={(0,0,0)}]; 
\else
 \path[3d cube/every face,3d cube/xz face100,unit xz plane={(0,1,0)}]; 
\fi
\ifnum\zproj>-1
 \path[3d cube/every face,3d cube/xy face100,unit xy plane={(0,0,1)}]; 
\fi
}},
pics/unit cube101/.style={code={
\path[get projections];
\draw(0,0,0) -- (1,1,1);
\ifnum\zproj=-1
 \path[3d cube/every face,3d cube/xy face101,unit xy plane={(0,0,0)}]; 
\fi
\ifnum\yproj=1
 \path[3d cube/every face,3d cube/yz face101,unit yz plane={(1,0,0)}]; 
\else
 \path[3d cube/every face,3d cube/yz face101,unit yz plane={(0,0,0)}]; 
\fi
\ifnum\xproj=1
 \path[3d cube/every face,3d cube/xz face101,unit xz plane={(0,0,0)}]; 
\else
 \path[3d cube/every face,3d cube/xz face101,unit xz plane={(0,1,0)}]; 
\fi
\ifnum\zproj>-1
 \path[3d cube/every face,3d cube/xy face101,unit xy plane={(0,0,1)}]; 
\fi
}},
pics/unit cube110/.style={code={
\path[get projections];
\draw(0,0,0) -- (1,1,1);
\ifnum\zproj=-1
 \path[3d cube/every face,3d cube/xy face110,unit xy plane={(0,0,0)}]; 
\fi
\ifnum\yproj=1
 \path[3d cube/every face,3d cube/yz face110,unit yz plane={(1,0,0)}]; 
\else
 \path[3d cube/every face,3d cube/yz face110,unit yz plane={(0,0,0)}]; 
\fi
\ifnum\xproj=1
 \path[3d cube/every face,3d cube/xz face110,unit xz plane={(0,0,0)}]; 
\else
 \path[3d cube/every face,3d cube/xz face110,unit xz plane={(0,1,0)}]; 
\fi
\ifnum\zproj>-1
 \path[3d cube/every face,3d cube/xy face110,unit xy plane={(0,0,1)}]; 
\fi
}},
pics/unit cube111/.style={code={
\path[get projections];
\draw(0,0,0) -- (1,1,1);
\ifnum\zproj=-1
 \path[3d cube/every face,3d cube/xy face111,unit xy plane={(0,0,0)}]; 
\fi
\ifnum\yproj=1
 \path[3d cube/every face,3d cube/yz face111,unit yz plane={(1,0,0)}]; 
\else
 \path[3d cube/every face,3d cube/yz face111,unit yz plane={(0,0,0)}]; 
\fi
\ifnum\xproj=1
 \path[3d cube/every face,3d cube/xz face111,unit xz plane={(0,0,0)}]; 
\else
 \path[3d cube/every face,3d cube/xz face111,unit xz plane={(0,1,0)}]; 
\fi
\ifnum\zproj>-1
 \path[3d cube/every face,3d cube/xy face111,unit xy plane={(0,0,1)}]; 
\fi
}},
3d cube/.cd,
xy face/.style={fill=blue!10},
xz face/.style={fill=gray!30},
yz face/.style={fill=gray!50},
xy face000/.style={draw=gray!10,line width=0mm,fill=gray!10},
xz face000/.style={draw=gray!30,line width=0mm,fill=gray!30},
yz face000/.style={draw=gray!50,line width=0mm,fill=gray!50},
xy face001/.style={draw=gray!10,line width=0mm,fill=gray!10},
xz face001/.style={draw=gray!30,line width=0mm,fill=gray!30},
yz face001/.style={draw=gray!50,line width=0mm,fill=gray!50},
xy face010/.style={draw=gray!10,line width=0mm,fill=gray!10},
xz face010/.style={draw=gray!30,line width=0mm,fill=gray!30},
yz face010/.style={draw=gray!50,line width=0mm,fill=gray!50},
xy face011/.style={draw=gray!10,line width=0mm,fill=gray!10},
xz face011/.style={draw=gray!30,line width=0mm,fill=gray!30},
yz face011/.style={draw=gray!50,line width=0mm,fill=gray!50},
xy face100/.style={draw=gray!10,line width=0mm,fill=gray!10},
xz face100/.style={draw=gray!30,line width=0mm,fill=gray!30},
yz face100/.style={draw=gray!50,line width=0mm,fill=gray!50},
xy face101/.style={draw=gray!10,line width=0mm,fill=gray!10},
xz face101/.style={draw=gray!30,line width=0mm,fill=gray!30},
yz face101/.style={draw=gray!50,line width=0mm,fill=gray!50},
xy face110/.style={draw=gray!10,line width=0mm,fill=gray!10},
xz face110/.style={draw=gray!30,line width=0mm,fill=gray!30},
yz face110/.style={draw=gray!50,line width=0mm,fill=gray!50},
xy face111/.style={draw=gray!10,line width=0mm,fill=gray!10},
xz face111/.style={draw=gray!30,line width=0mm,fill=gray!30},
yz face111/.style={draw=gray!50,line width=0mm,fill=gray!50},
xy face2/.style={fill=gray!10},
xz face2/.style={fill=gray!30},
yz face2/.style={fill=gray!50},
num cubes x/.estore in=\NumCubesX,
num cubes y/.estore in=\NumCubesY,
num cubes z/.estore in=\NumCubesZ,
num cubes x=1,num cubes y/.initial=1,num cubes z/.initial=1,
cube scale/.initial=1,
every face/.style={draw,very thick},
/tikz/pics/.cd,
cube array/.style={code={%
 \tikzset{3d cube/.cd,#1}
  \path[get projections];
  \ifnum\yproj=1
   
  \else 
   \ifnum\NumCubesX>1
    \pgfmathtruncatemacro{\NextToLast}{\NumCubesX-1}
    
   \else
       
   \fi 
  \fi
  \ifnum\xproj=-1
   
  \else 
   \ifnum\NumCubesY>1
    \pgfmathtruncatemacro{\NextToLast}{\NumCubesX-1}
    
   \else
       
   \fi 
  \fi
  \ifnum\zproj=1
   
  \else 
   \ifnum\NumCubesZ>1
    \pgfmathtruncatemacro{\NextToLast}{\NumCubesX-1}
    
   \else
       
   \fi 
   
  \fi

}}
}

\begin{document}
\baselineskip=17pt

 \maketitle
%\frontmatter
% eigene Definitionen
\newtheorem{theorem}{Theorem}
\newtheorem*{klar}{Klar}
\newtheorem{method}{Method}

\newtheorem{lemma}{Lemma}

\newtheorem{cor}{Corollary}

\newtheorem{conjecture}{Conjecture}

\theoremstyle{definition}
\newtheorem{defi}{Definition}

\newtheorem{bsp}{Beispiel}

\newtheorem*{bem}{Bemerkung}

\newtheorem*{vorschau}{Vorschau}

\newtheorem*{erg}{Ergänzung}

\theoremstyle{remark}
\newtheorem{remark}{Remark}

\newtheorem*{notation}{Notation}

\newtheorem{claim}{Claim}[theorem]
\renewcommand{\theclaim}{\arabic{claim}}
\newenvironment{proofofclaim}[1][\proofname\ of Claim \theclaim]{%
  \proof[#1]%
  \renewcommand\qedsymbol{$\blacksquare$}
}{\endproof}

\newcommand{\ndiv}{\not \hspace{3pt} \mid }

\newcommand*\hexbrace[2]{%
  \underset{#2}{\underbrace{\rule{#1}{0pt}}}}

\section*{Abstract}
\label{cha:abstract}

We construct a unilateral lattice tiling of $\mathbb{R}^n$ into hypercubes of two differnet side lengths $p$ or $q$. This generalizes the Pythagorean tiling in $\mathbb{R}^2$. We also show that this tiling is unique up to symmetries, which proves a variation of a conjecture by B\"olcskei from 2001. For positive integers $p$ and $q$ this tiling also provides a tiling of $(\mathbb{Z}/(p^n+q^n)\mathbb{Z})^n$.

\section{Introduction}

In 1907 Minkowski \cite{minkowski1907diophantische} conjectured that any tiling of $\mathbb{R}^n$ into $n$-dimensional unit hypercubes whose center points form a lattice has to contain two hypercubes that share a full facet. Keller \cite{keller1930uber} generalized this conjecture by allowing any tiling of $\mathbb{R}^n$ into $n$-dimensional unit hypercubes. Perron \cite{perron1940luckenlose} proved Keller's conjecture for $n\leq 6$ in 1940 and Hajos \cite{hajos1942einfache} proved Minkowski's original conjecture in 1942.

In 1992 Lagarias and Shor \cite{lagarias1992keller} proved Keller's conjecture to be false, in particular they showed it is false for $n\geq 10$. The gaps were filled by Mackey \cite{mackey2002cube} for dimensions $8$ and $9$ and Brakensiek et. al. \cite{brakensiek2020resolution} for dimension $7$, showing that $7$ is the smallest dimension where Keller's conjecture is true. For additional information see the survey on unit cubes by Zong \cite{zong2005known}.

In $\mathbb{R}^2$ there are some related results:
\begin{itemize}
\item The well known Pythagorean tiling (Figure \ref{phytil}) is a lattice tiling of $\mathbb{R}^2$ into squares of two sizes as well as the unique unilateral and equitransitive tiling into squares of two sizes (\cite{grunbaum1987tilings}, \cite{martini}).

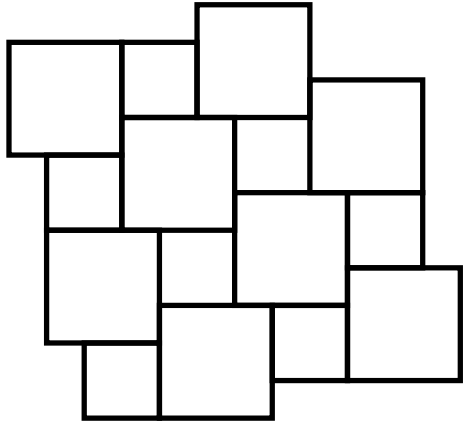
\begin{figure}

\begin{tikzpicture}[yscale=1,scale=0.5]

\begin{scope}

\begin{scope}[xshift=3cm,yshift=1cm]

\draw[-, line width=2pt] (0,0) -- (0,2) -- (2,2) -- (2,0) -- cycle;
\end{scope}

\begin{scope}[xshift=2cm,yshift=6cm]

\draw[-, line width=2pt] (0,0) -- (0,2) -- (2,2) -- (2,0) -- cycle;
\end{scope}

\begin{scope}[xshift=8cm,yshift=2cm]

\draw[-, line width=2pt] (0,0) -- (0,2) -- (2,2) -- (2,0) -- cycle;
\end{scope}

\begin{scope}[xshift=10cm,yshift=5cm]

\draw[-, line width=2pt] (0,0) -- (0,2) -- (2,2) -- (2,0) -- cycle;
\end{scope}

\begin{scope}[xshift=4cm,yshift=9cm]

\draw[-, line width=2pt] (0,0) -- (0,2) -- (2,2) -- (2,0) -- cycle;
\end{scope}

\begin{scope}[xshift=2cm,yshift=3cm]

\draw[-,line width=2pt] (0,0) -- (0,3) -- (3,3) -- (3,0) -- cycle;
\end{scope}

\begin{scope}[xshift=4cm,yshift=6cm]

\draw[-, line width=2pt] (0,0) -- (0,3) -- (3,3) -- (3,0) -- cycle;
\end{scope}

\begin{scope}[xshift=6cm,yshift=9cm]

\draw[-, line width=2pt] (0,0) -- (0,3) -- (3,3) -- (3,0) -- cycle;
\end{scope}

\begin{scope}[xshift=5cm,yshift=1cm]

\draw[-,line width=2pt] (0,0) -- (0,3) -- (3,3) -- (3,0) -- cycle;
\end{scope}

\begin{scope}[xshift=7cm,yshift=4cm]

\draw[-, line width=2pt] (0,0) -- (0,3) -- (3,3) -- (3,0) -- cycle;
\end{scope}

\begin{scope}[xshift=9cm,yshift=7cm]

\draw[-, line width=2pt] (0,0) -- (0,3) -- (3,3) -- (3,0) -- cycle;
\end{scope}

\begin{scope}[xshift=10cm,yshift=2cm]

\draw[-, line width=2pt] (0,0) -- (0,3) -- (3,3) -- (3,0) -- cycle;
\end{scope}

\begin{scope}[xshift=1cm,yshift=8cm]

\draw[-,line width=2pt] (0,0) -- (0,3) -- (3,3) -- (3,0) -- cycle;
\end{scope}

\end{scope}

\end{tikzpicture}
\caption{The Pythagorean tiling}
\label{phytil}
\end{figure}

\begin{figure}

\foreach \Angle in {50} 
{\tdplotsetmaincoords{60}{\Angle} % the first argument cannot be larger than 90

\begin{tikzpicture}[line join=round,font=\sffamily,3d cube/.cd,
num cubes x=1,num cubes y=1,num cubes z=1]

 \path[use as bounding box] (-1,-2) rectangle (8,6);
\begin{scope}[local bounding box=first row]

 \begin{scope}[tdplot_main_coords,local bounding box=array]
 \input{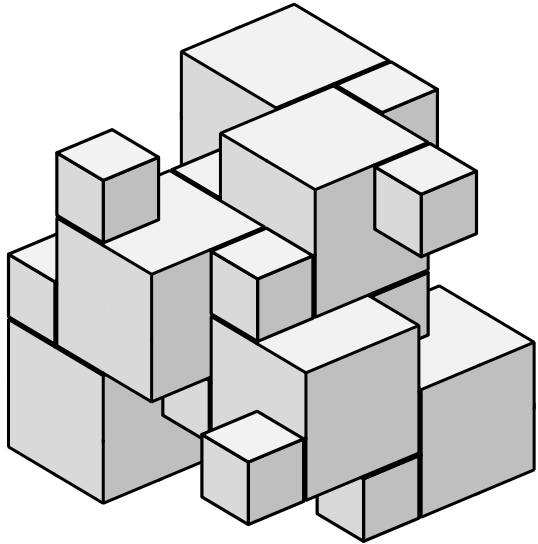}
 \end{scope}

\end{scope}

 {
}

\end{tikzpicture}}
\caption{Roger's filling}
\label{rogfil}
\end{figure}

\item There are exactly eight different unilateral and equitransitive tilings of $\mathbb{R}^2$ into squares of three sizes (\cite{bolcskei2000classification}, \cite{grunbaum1987tilings}, \cite{martini}, \cite{schattschneider2000unilateral}).

\item $\mathbb{R}^2$ can be tiled into squares of pairwise distinct integer side lengths in $2^{\aleph_0}$ ways. For $n\geq 3$ this is not possible not even if we only require neighbouring cubes to have different sizes (\cite{dawson1984filling}).

\item Sprague \cite{sprague1939beispiel} gave the first example of a tiling of a square into squares of pairwise different integer side lengths (see also \cite{MR511994} and \cite{grunbaum1987tilings}).

\item Meir and Moser \cite{MEIR1968126} asked whether a square of area $\sum_{n=1}^\infty 1/n^2=\pi^2/6$ can be tiled by the squares of side length $1/n$ for $n\in\mathbb{Z}_{\geq 1}$. The problem remains open but Januszewski and Zielonka \cite{januszewski2021note} showed that for $1/2 < t \leq 2/3$, the square of area $\sum_{n=1}^\infty 1/n^{2t}$ can be tiled by the squares of side length $1/n^t$ for $n\in\mathbb{Z}_{\geq 1}$ and Tao \cite{tao2022perfectly} showed that for every $2/3 < t < 1$ there exists a $n_0\in\mathbb{N}$ such that the square of area $\sum_{n=n_0}^\infty 1/n^{2t}$ can be tiled by the squares of side length $1/n^t$ for $n\in\mathbb{Z}_{\geq n_0}$.
\end{itemize}

In order to state the main result, we need the following definitons: A \emph{tiling} is a partition into closed sets without holes, called tiles, whose pairwise intersection is of Lebesgue measure zero. We further call sets of Lebesgue measure zero \emph{null sets}. For tilings of $\mathbb{R}^n$ into hypercubes, we call a tiling \emph{unilateral} if no two hypercubes of the same size share a full facet and \emph{equitransitive} if any two hypercubes of the same size can be mapped to each other by an isomorphism that keeps the tiling unchanged. \emph{Lattice tilings} are tilings where the center points of all hypercubes of the same size form a lattice and are therefore also equitransitive.

In 2001, B\"olcskei \cite{MR1874434} proved that Roger's filling (Figure \ref{rogfil}) is the only unilateral and equitransitive tiling of $\mathbb{R}^3$ into cubes of two sizes and conjectured this to be true for any higher dimension. 

\begin{conjecture}[B\"olcskei \cite{MR1874434}]
\label{conj}
In every dimension $d$ there exists precisely one equitransitive
unilateral tiling by $d$-dimensional cubes of two sizes.
\end{conjecture}

We prove the following variation of the conjecture, where we consider lattice tiling instead of equitransitive tilings:

\begin{theorem}
\label{main}
Let $p,q\in\mathbb{R}_+$ with $q<p$ and $n\in\mathbb{N}_{\geq2}$. There exists exactly one unilateral lattice tiling of $\mathbb{R}^n$ into hypercubes of side length $p$ or $q$ up to symmetries.
\end{theorem}

\begin{remark}
Since lattice tilings are equitransitive, in regards to existence Theorem \ref{main} is stronger than Conjecture \ref{conj} and in regards to uniqueness it is the other way around.
\end{remark}

\begin{notation}
We write $[a,b]$ for $\{a,a+1,a+2,...,b-1,b\}\subseteq\mathbb{Z}$ and $e_i$ for the $i$-th standard unit vector in $\mathbb{R}^n$.
\end{notation}

\section{Existence}

We split Theorem \ref{main} into two parts: The first is Theorem \ref{thmint1} and the other is dealt with in section \ref{secunique}.

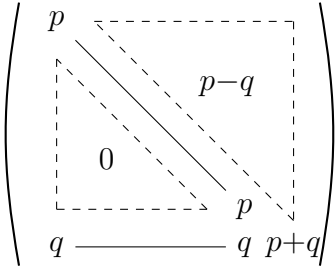
\begin{theorem}
\label{thmint1}
Let $p,q\in\mathbb{R}_+$ with $q<p$ and $n\in\mathbb{N}_{\geq2}$. There exists a lattice tiling of $\mathbb{R}^n$ into hypercubes of side length $p$ or $q$.
\end{theorem}

\begin{remark}
After a preprint of this article appeared on arXiv, the author was informed by Mihalis Kolountzakis that Theorem \ref{thmint1} follows from \autocite[Theorem 8]{MR1605065}.
\end{remark}

Let $A$ be the following $n\times n$ matrix:

\vspace{0.3cm}

\begin{tikzpicture}[yscale=-1,scale=0.5]

\node  at (1,1) {$q$};
\node  at (7,7) {$q$};
\node  at (1,2) {$-p$};
\node  at (6,7) {$-p$};
\node  at (7,1) {$p$};

\node  at (2.33,5.66) {$0$};
\node  at (5.33,3) {$0$};

\draw[-] (1.5,1.5) -- (6.5,6.5);
\draw[-] (1.5,2.5) -- (5.5,6.5);

\draw[-,dashed] (1,3) -- (1,7);
\draw[-,dashed] (1,7) -- (5,7);
\draw[-,dashed] (5,7) -- (1,3);

\draw[-,dashed] (2,1) -- (6,1);
\draw[-,dashed] (6,1) -- (6,2);
\draw[-,dashed] (6,2) -- (7,2);
\draw[-,dashed] (7,2) -- (7,6);
\draw[-,dashed] (7,6) -- (2,1);

\draw[-,thick] (0,0.5) to[out=100, in=-100] (0,7.5);
\draw[-,thick] (8,0.5) to[out=80, in=-80] (8,7.5);

\end{tikzpicture}

\vspace{0.3cm}

Theorem \ref{thmint1} follows directly from the following Lemma.

\begin{lemma}
\label{lem1}
$[0,q]^n\cup [0,p]^{n-1}\times [q,q+p]$ is a system of representatives of $\mathbb{R}^n/A\mathbb{Z}^n$ up to null sets.
\end{lemma}

\begin{proof}

We split the proof of Lemma \ref{lem1} into four claims.

\begin{claim}
Every element in $\mathbb{R}^n$ has a representative in $[0,q)^{n-1}\times \mathbb{R}$.
\end{claim}

\begin{proofofclaim}
We can inductively add multiples of the $k$-th colomn of $A$ to find a representative of any element in $[0,q)^{k}\times \mathbb{R}^{n-k}$ up to $k=n-1$.

\end{proofofclaim}

\begin{claim}
Every element $x=(x_1,x_2,...,x_n)^T \in [0,q)^{n-1}\times \mathbb{R}_{\geq 0}$ with $x_n\geq p+q$ has a representative in $[0,q)^{n-1}\times [0,x_n-q]$.
\end{claim}

\begin{proofofclaim}
Consider the following matrix $A_1$: 

 \vspace{0.3cm}

\begin{tikzpicture}[yscale=-1,scale=0.5]

\node  at (1,1) {$p$};
\node  at (7.3,6.95) {$p{+}q$};
\node  at (6,6) {$p$};
\node  at (1,7) {$q$};
\node  at (6,7) {$q$};

\node  at (2.33,4.66) {$0$};
\node  at (5.53,2.66) {$p{-}q$};

\draw[-] (1.5,1.5) -- (5.5,5.5);
\draw[-] (1.5,7) -- (5.5,7);

\draw[-,dashed] (1,2) -- (1,6);
\draw[-,dashed] (1,6) -- (5,6);
\draw[-,dashed] (5,6) -- (1,2);

\draw[-,dashed] (2,1) -- (7.3,1);
\draw[-,dashed] (7.3,1) -- (7.3,6.3);
\draw[-,dashed] (7.3,6.3) -- (2,1);

\draw[-,thick] (0,0.5) to[out=100, in=-100] (0,7.5);
\draw[-,thick] (8,0.5) to[out=80, in=-80] (8,7.5);

\end{tikzpicture}
 \vspace{0.3cm} 
    
   Let $a_1,..,a_n$ be the columns of $A$ and $b_1,...,b_n$ the columns of $A_1$. We obtain $A_1$ in the following way:
   \begin{itemize}
   \item $b_1=a_n$
   \item $b_k=b_{k-1}-a_{k-1}$, for $k\in[2,n]$
   \end{itemize}
   
   $$b_k=b_{k-1}-a_{k-1}=\begin{pmatrix}  p-q\\ 
   \vdots \\ 
    p-q\\
    p \\
    0\\
    0\\
    \vdots \\
    0 \\
    q
   \end{pmatrix}-\begin{pmatrix}  0\\ 
   \vdots \\ 
    0\\
    q \\
    -p\\
    0 \\
    \vdots \\
    0 \\
    0
   \end{pmatrix}=\begin{pmatrix}  p-q\\ 
   \vdots \\ 
    p-q\\
    p-q \\
    p\\
    0 \\
    \vdots \\
    0 \\
    q
   \end{pmatrix}\begin{array}{l}
  \\[-45mm] \rdelim\}{3}{120mm}[$k-2$] \\
\end{array}[-1ex] $$

  Now let $l\in[1,n]$ such that $x_i<p$ for all $i<l$ and $x_l\geq p$, then $x-b_l$ is a representative of $x$ in $[0,q)^{n-1}\times [0,x_n-q]$.

\end{proofofclaim}

\begin{claim}
Every element $x=(x_1,x_2,...,x_n)^T \in [0,q)^{n-1}\times \mathbb{R}_{<0}$ has a representative in $[0,q)^{n-1}\times [x_n+q,p+q)$.
\end{claim}

\begin{proofofclaim}
Let $b_1,...b_n$ be as before and let $l\in[1,n]$ be such that $x_i\geq q-p$ for all $i<l$ and $x_l<q-p$. Then $x+b_l$ is a representative of $x$ in $[0,q)^{n-1}\times [x_n+q,p+q)$.

\end{proofofclaim}

By iteratively applying the last two claims we get a representative in $[0,q)^{n-1}\times [0,p+q)$ for every element.

Now let $x\in([0,q)^{n-1}\times [0,p+q))\setminus ([0,q)^n\cup [0,p)^{n-1}\times [q,q+p))$, and again let $l\in[1,n]$ such that $x_i<p$ for all $i<l$ and $x_l\geq p$. From the definition of $x$ it follows that $l<n$ and $x-b_l$ is a representative of $x$ in $[0,q)^{n}$ because $x_n>q$.

\begin{claim}
The following holds: $\det (A)=p^n+q^n$.
\end{claim}

\begin{proofofclaim}
We calculate the determinant with the Laplace expansion along the first row.

$$\det(A)=q\begin{vmatrix}
q &  &  & \\
-p & \ddots &  &  \\
 & \ddots & \ddots &  \\
 &  & -p & q
\end{vmatrix}
+(-1)^{n-1} p\begin{vmatrix}
-p & q &  & \\
 & \ddots & \ddots &  \\
 &  & \ddots & q \\
 &  &  & -p
\end{vmatrix}=q^n+p^n$$

\end{proofofclaim}

We now prove Lemma \ref{lem1}. The parallelepiped $P$ spanned by the columns of $A$ is a system of representatives (up to null sets) of $\mathbb{R}^n/A\mathbb{Z}^n$ with Lebesgue measure $\det(A)=p^n+q^n$. Let $P=\biguplus_{i\in\mathbb{N}} P_i $ be a decomposition of $P$ into sets of points that share the same translation into $C:=[0,q)^n\cup [0,p)^{n-1}\times [q,q+p)$ by the previous reductions and let $Q_i$ be the corresponding translated sets. Every set $P_i$ and therefore also $Q_i$ is measurable since it is constructed by intersections of halfspaces. Since $P$ is a system of representatives (up to null sets), the sets $Q_i$ have pairwise intersection of measure $0$ and  therefore $$p^n+q^n=\lambda(P)=\sum_{i\in\mathbb{N}}\lambda(P_i)=\sum_{i\in\mathbb{N}}\lambda(Q_i)\leq \lambda(C)=p^n+q^n.$$ Consequently, $[0,q]^n\cup [0,p]^{n-1}\times [q,q+p]=C=\bigcup_{i\in\mathbb{N}} Q_i $ up to null sets, which concludes the proof.
\end{proof}

\begin{remark}
Two different points in the interior of $[0,q]^n\cup [0,p]^{n-1}\times [q,q+p]$ cannot be representatives of the same element in $\mathbb{R}^n/A\mathbb{Z}^n$ as this would imply that there were neighbourhoods of the two points corresponding to the same set in  $\mathbb{R}^n/A\mathbb{Z}^n$ of non-zero Lebesgue measure.
\end{remark}

To show that the tiling is unilateral the followimg Lemma is sufficient.

\begin{lemma}
Let $i\in [1,n]$ and let $e_i$ be the $i$-th standard unit vector, then $pe_i\not\in A\mathbb{Z}^n$ and $qe_i\not\in A\mathbb{Z}^n$.
\end{lemma}

\begin{proof}
There are two points in the interior of $[0,q]^n$ with difference $pe_i$ and therefore they are both different representatives of  $\mathbb{R}^n/A\mathbb{Z}^n$, so $pe_i\not\in A\mathbb{Z}^n$. Analogously, there are two points in the interior of $[0,q)^n\cup [0,p)^{n-1}\times [q,q+p)$ with difference $qe_n\not\in A\mathbb{Z}^n$.
Suppose $qe_i\in A\mathbb{Z}^n$  for $i<n$ then also $qe_i-a_i=pe_{i+1}\in A\mathbb{Z}^n$, a contradiction to the previous statement.
\end{proof}

\section{Symmetries and periodicity}

The group of symmetries of the tiling has to keep the cube invariant, so it is a supgroup of the hyperoctahedral group $B_n$ which has order $2^n n!$. Again we work with the matrix formulation of the problem and write the group operations as matrix multiplications from the left. Here the hyperoctahedral group can be written in the following way: $B_n\cong B'_n$ where $$B'_n:=\{S\in\{0,\pm 1\}^{n\times n}|\text{every row and column contains exactly one $\pm 1$ entry}\}$$
and we find the biggest supgroup of $B'_n$ which fixes $A\mathbb{Z}^n$.
In other words $S\in B'_n$ is in the stabilizer of the tiling $\iff (SA)X=A$ has an integer solution $$\iff A^{-1}S^{-1}A\in\mathbb{Z}^{n \times n}.$$

\begin{theorem}
$S:=(s_{i,j})_{i,j\in [1,n]}\in B'_n$ is in the stabilizer of the tiling $\iff$ $S$ is of the following form:
\begin{itemize}
\item $s_{i,j}=s_{i+1,j+1}$ for $i<n,j<n$
\item $s_{i,n}=-s_{i+1,1}$ and $s_{n,j}=-s_{1,j+1}$ for $i<n,j<n$
\item $s_{n,n}=s_{1,1}.$
\end{itemize}

\end{theorem}
 
 \vspace{0.3cm}

\begin{tikzpicture}[yscale=-1,scale=0.5]

\node  at (1,4) {$\pm 1$};
\node  at (4,7) {$\pm 1$};
\node  at (5,1) {$\mp 1$};
\node  at (7,3) {$\mp 1$};

\node  at (4.166,3.833) {$0$};

\draw[-] (1.5,4.5) -- (3.5,6.5);
\draw[-] (5.5,1.5) -- (6.5,2.5);

\draw[-,dashed] (1,1) -- (1,3);
\draw[-,dashed] (1,3) -- (5,7);
\draw[-,dashed] (5,7) -- (7,7);
\draw[-,dashed] (7,7) -- (7,4);
\draw[-,dashed] (7,4) -- (4,1);
\draw[-,dashed] (4,1) -- (1,1);

\draw[-,dashed] (1,5) -- (1,7);
\draw[-,dashed] (1,7) -- (3,7);
\draw[-,dashed] (3,7) -- (1,5);

\draw[-,dashed] (6,1) -- (7,1);
\draw[-,dashed] (7,1) -- (7,2);
\draw[-,dashed] (7,2) -- (6,1);

\draw[-,thick] (0,0.5) to[out=100, in=-100] (0,7.5);
\draw[-,thick] (8,0.5) to[out=80, in=-80] (8,7.5);

\end{tikzpicture}
 \vspace{0.3cm}
    
    \begin{proof}

"$\Leftarrow$": Let $S$ be as above and $i$ such that $s_{i,1}=\pm 1$. W.l.o.g. $s_{i,1}= 1$.
Now $$Sa_j=\begin{cases}
qe_{i+j+1}-pe_{i+j}=a_{i+j-1} & j\leq n-i \\
qe_{n}+pe_{1}=a_n & j=n-i+1 \\
-qe_{-n+i+j-1}+pe_{-n+i+j}=-a_{-n+i+j-1} & j\geq n-i+2.
\end{cases}$$

"$\Rightarrow$": Let $v:=qe_i\pm pe_j$ such that $v\neq a_i$ and $i\neq j$. Then $$a_i-v=\begin{cases}
-pe_{i+1}\mp pe_j & i\leq n-1 \land j  \neq i+1\\
pe_{1}\mp pe_j & i= n \land j\neq 1\\
-2pe_{j} & i\leq n-1 \land j=i+1\\
2pe_{1} & i=n \land j=1.\\
\end{cases}$$
In the first two cases $a_i-v$ is not in $A\mathbb{Z}^n$ because it is a vector connecting two inner points of $[0,q)^n$. Since $a_j-2pe_j$ is also connecting two inner points of $[0,q)^n$, $a_i-v\not\in A\mathbb{Z}^n$ and therefore $v\not\in A\mathbb{Z}^n$. Since every vector $Sa_j$ is of the form $qe_k\pm pe_l$ or the negation of it for all $ j\in [1,n]$ there exists $i\in [1,n]$ such that $Sa_j=\pm a_i$. Now suppose $s_{i,j}=\pm 1$, then $Sa_j=\pm a_i$ and therefore 
 $$s_{i+1,j+1}=\begin{cases}
\pm 1 & i\leq n-1 \land j\leq n-1\\
\mp 1 & i=n  \land j\leq n-1\\
\mp 1 & i\leq n-1  \land j=n\\
\pm 1 & i=n  \land j=n\\
\end{cases}$$
\end{proof}

On the periodicity of the tiling along the coordinate axis we have the following result:

\begin{theorem}
Let $p,q\in\mathbb{N}$ with $\gcd(p,q)=1$. Then $p^n+q^n=\min\{l\in\mathbb{N} | le_i\in A\mathbb{Z}^n\}$ for all $i\in[1,n]$. In particular, there exists a unilateral lattice tiling of $(\mathbb{Z}/(p^n+q^n)\mathbb{Z})^n$ into hypercubes of side length $p$ or $q$.
\end{theorem}

\begin{proof}
Let $i\in[1,n]$. From Cramer's rule we know that the solution of $Ax=e_i$ in $\mathbb{Q}^n$ is of the form $x=k/\det(A)=k/(p^n+q^n)$ for a $k=(k_1,...,k_n)^T\in\mathbb{Z}^n$ and therefore $(p^n+q^n)e_i\in A\mathbb{Z}^n$. More specifically, $k_i=\det(A_i)/(p^n+q^n)$, where $A_i=(a_1,...,a_{i-1},e_i,a_{i+1},...,a_n)$. Calculating the determinant with the Laplace expansion along the last column we get $\det(A_i)=q^{n-1}$ and since $\gcd(p^n+q^n,q^{n-1})=1$ the statement follows.

\begin{remark}
For $q=2$ and $p=1$ the tiling gives a lower bound on the maximal number of hypercubes of side length two that can be packed in $(\mathbb{Z}/(2^n+1)\mathbb{Z})^n$, which coincides with the optimal solution (see \cite{baumert1971combinatorial}). Finding such packings is used to determine the Shannon capacity of odd cycles (see \cite{Lovasz},\cite{mathew2017new},\cite{Shannon}).

\end{remark}

\end{proof}

\section{Uniqueness}
\label{secunique}

We now show that every unilateral lattice tiling of $\mathbb{R}^n$ into hypercubes of two sizes is equaivalent to the tiling described in the previous sections. It is known that such a tiling does not exist if we only use cubes of one size (see Haj{\'o}s \cite{hajos1942einfache}) and so we can describe the tiling by means of two cubes of different size, translated by $B\mathbb{Z}^n$ for an $n\times n$ matrix $B$ of full rank.
\begin{lemma}
\label{smallcubes}
Two small cubes cannot touch.
\end{lemma}

\begin{proof}
Assume the contrary. Then there exists $x:=(x_1,...,x_n)^T\in B\mathbb{Z}^n$ with $|x_i|\leq p$ for all $i \in [1,n]$. Now $x$ connects two inner points of the big cube, a contradiction.
\end{proof}

\begin{lemma}
A small and a big cube that properly touch share a corner.
\end{lemma}

\begin{proof}
Let $T$ be a big cube and $S$ a small cube that properly touch. W.l.o.g. let $T=[0,q]^n$ and $S$ properly touches $T$ in $\{q\}\times \mathbb{R}^{n-1}$. We write $S=[q,q+p]\times \bigotimes_{j=2}^{n}[x_i,x_i+p]$ where $x_i\in (-p,p+q)$ for all $i\in[2,n]$.

Assume $S$ touches $T$ in a full side but $S$ and $T$ do not share a corner then $x_i\in [0,q-p]$ for all $i\in[2,n]$ and w.l.o.g. $x_n\not\in \{0,p-q\}.$ There have to be big cubes $T_1,T_2$ that touch $T$ in $\{q\}\times {\bigotimes_{j=2}^{n-1}[x_i,x_i+\epsilon]}\times [x_n-\epsilon,x_n]$ and $\{q\}\times {\bigotimes_{j=2}^{n-1}[x_i,x_i+\epsilon]}\times [x_n+q,x_n+q+\epsilon]$ respectively for a $\epsilon>0$. Therefore $T_1= [q,2q]\times \bigotimes_{j=2}^{n-1}[y_i,y_i+q]\times [x_n-p,x_n]$ and $T_2= [q,2q]\times \bigotimes_{j=2}^{n-1}[z_i,z_i+q]\times [x_n+p,x_n+p+q]$ where $y_i,z_i\in (x_i-q,x_i+\epsilon+q)$ for all $i\in[2,n-1]$.
Now $E:=[q+p,2q]\times\bigotimes_{j=2}^{n-1}[x_i,x_i+\epsilon]\times[x_n,x_n+p]$ cannot be filled by any cubes, because it touches a small cube and it touches $T_1$ and $T_2$ in opposite full sides of distance $p$ (see Figure \ref{fig1}).

Now assume that $S$ does not touch $T$ in a full side. W.l.o.g let $x_n>q-p$ and let $x_i\in[0,q)$ for all $i\in[2,n-1]$. Let $T_3$ be the big cube properly touching $S$ in $\{q\}\times\bigotimes_{j=2}^{n-1}[x_i,x_i+\epsilon]\times [q,q+\epsilon]$ and let $T_4$ be the big cube properly touching $S$ in $\{q+p\}\times\bigotimes_{j=2}^{n}[x_i,x_i+\epsilon]$ for a $\epsilon>0$. We write $T_4= [q+p,2q+p]\times \bigotimes_{j=2}^{n}[w_i,w_i+q]$ with $w_i\in [x_i-q+\epsilon,x_i]$ for all $i\in[2,n-1]$. If $w_n<x_n$ then $E_1:={[q,q+p]}\times \bigotimes_{j=2}^{n-1}[x_i,x_i+\epsilon]\times[w_n,x_n]$ cannot be filled by any cubes, because it touches a small cube and it touches $T$ and $T_4$ in opposite full sides of distance $p$. Otherwise $E_2:={[q,q+p]}\times \bigotimes_{j=2}^{n-1}[x_i,x_i+\epsilon]\times[q,x_n+p]$ cannot be filled by any cubes, because it touches a small cube and it touches $T_3$ and $T_4$ in opposite full sides of distance $p$ (see Figure \ref{fig2}).
\end{proof}

\begin{figure}[H]
\begin{minipage}[c]{0.4\linewidth}
\centering

\begin{tikzpicture}[scale=0.7]
\fill[color=white] (-3,0) -- (-3,8) -- (3,8) -- (3,0) -- cycle;
\begin{scope}[xshift=0cm,yshift=0cm]
\draw[-,line width=2pt] (0,0) -- (0,3) -- (3,3) -- (3,0) -- cycle;
\node  at (1.5,1.5) {$T$};
\end{scope}

\begin{scope}[xshift=-2.7cm,yshift=3cm]
\draw[-, line width=2pt] (0,0) -- (0,3) -- (3,3) -- (3,0) -- cycle;
\node  at (1.5,1.5) {$T_1$};
\end{scope}

\begin{scope}[xshift=2cm,yshift=3cm]
\draw[-, line width=2pt] (0,0) -- (0,3) -- (3,3) -- (3,0) -- cycle;
\node  at (1.5,1.5) {$T_2$};
\end{scope}

\begin{scope}[xshift=0.3cm,yshift=3cm]
\draw[-, line width=2pt] (0,0) -- (0,1.7) -- (1.7,1.7) -- (1.7,0) -- cycle;
\node  at (0.85,0.85) {$S$};
\end{scope}

\draw[dotted, line width=2pt] (0.3,4.7) -- (0.3,6) -- (2,6) -- (2,4.7) -- cycle;
\node  at (1.15,5.35) {$E$};

\end{tikzpicture}
\caption{}
\label{fig1}
\end{minipage}\hfill
\begin{minipage}[c]{0.4\linewidth}
\centering
\begin{tikzpicture}[scale=0.7]
\fill[color=white] (0,0) -- (0,8) -- (6,8) -- (6,0) -- cycle;

\begin{scope}[xshift=0cm,yshift=0cm]
\draw[-,line width=2pt] (0,0) -- (0,3) -- (3,3) -- (3,0) -- cycle;
\node  at (1.5,1.5) {$T$};
\end{scope}

\begin{scope}[xshift=3cm,yshift=0cm]
\draw[-, line width=2pt] (0,0) -- (0,3) -- (3,3) -- (3,0) -- cycle;
\node  at (1.5,1.5) {$T_3$};
\end{scope}

\begin{scope}[xshift=1.7cm,yshift=4.7cm]
\draw[-, line width=2pt] (0,0) -- (0,3) -- (3,3) -- (3,0) -- cycle;
\node  at (1.5,1.5) {$T_4$};
\end{scope}

\begin{scope}[xshift=2.4cm,yshift=3cm]
\draw[-, line width=2pt] (0,0) -- (0,1.7) -- (1.7,1.7) -- (1.7,0) -- cycle;
\node  at (0.85,0.85) {$S$};
\end{scope}

\draw[dotted, line width=2pt] (1.7,3) -- (1.7,4.7) -- (2.4,4.7) -- (2.4,3) -- cycle;
\node  at (2.05,3.65) {$E_1$};

\draw[dotted, line width=2pt] (4.1,3) -- (4.1,4.7) -- (4.7,4.7) -- (4.7,3) -- cycle;
\node  at (4.4,3.65) {$E_2$};

\end{tikzpicture}
\caption{}
\label{fig2}
\end{minipage}
\end{figure}

\begin{lemma}
\label{opposite}
Let $S$ be a small cube with center $s$ and T a big cube with center $s+t$ that properly touches $S$. Then there is a  big cube with center $s-t$ in the tiling.
\end{lemma}

\begin{proof}
W.l.o.g let $S:=[0,p]^n$ and $T:=[-q,0]\times [0,q]^{n-1}$ and let $T'$ be the big cube  properly touching $S$ in the hyperplane $\{p\}\times \mathbb{R}^{n-1}$ and write $T'=[p,p+q]\times \bigotimes_{i=2}^n I_i$ with $I_i\in \{[0,q],[p-q,p]\}$. Asume to the contrary that $T'\neq [p,p+q]\times [p-q,p]^{n-1}$ and let $l\in[2,n]$ such that $I_l=[0,q]$. Then $E:=[0,p]^{l-1}\times [p,q] \times [0,p]^{n-l}$ touches $T$ and $T'$ with a full side in parallel hyperplanes with distance $p$, so it can only be filled with small cubes. This contradicts Lemma \ref{smallcubes} because $E$ touches $S$. 
\end{proof}

Now assume $T:=[0,q]^n$ and $S:=[0,p]^{n-1}\times[q,q+p]$ are two cubes in the tiling. Let $T_i$ be the cube properly touching $S$ in the hyperplane $\mathbb{R}^{i-1}\times\{p\}\times\mathbb{R}^{n-i}$. Then $T_i=\bigotimes_{j=1}^{i-1} I_i^{(j)}\times [p,p+q]\times\bigotimes_{j=i+1}^{n-1} I_i^{(j)}\times [q,2q]$ with $I_i^{(j)}\in \{[0,q],[p-q,p]\}$. Since $T_i$ and $T_j$ do not properly intersect we have $I_i^{(j)}=[p-q,p]$ or $I_j^{(i)}=[p-q,p]$. From Lemma \ref{opposite} we know that $T'_i:=\bigotimes_{j=1}^{i-1} J_i^{(j)}\times [-q,0]\times\bigotimes_{j=i+1}^{n-1} J_i^{(j)}\times [p,p+q]$ with  $$J_i^{(j)}=\begin{cases}
[0,q] &I_i^{(j)}=[p-q,p]\\
[p-q,p] & I_i^{(j)}=[0,q]\\
\end{cases}$$ is also a cube in the tiling. Since $T'_i$ and $T_j$ cannot properly intersect,  
$J_i^{(j)}=[p-q,p]$ or $I_j^{(i)}=[0,q]$ and consequently $\{I_i^{(j)},I_j^{(i)}\}= \{[0,q],[p-q,p]\}$.

\begin{lemma}
\label{uniqueneighbour}
There exists exactly one $i\in [1,n-1]$ such that $T_i=\bigotimes_{j=1}^{i-1} [0,q]\times {[p,p+q]}\times\bigotimes_{j=i+1}^{n-1} [0,q]\times [q,2q]$.
\end{lemma}

\begin{proof}
From Lemma \ref{opposite} we know that $T_n:=[p-q,p]^{n-1}\times[q+p,2q+p]$ is also a cube in the tiling.
Let $T'$ be such that $T'$ touches $T$ in $[q-\epsilon,q]^{n-1}\times\{q\}$ then $T'=\bigotimes_{j=1}^{n-1} [x_i,x_i+q] \times [q,2q]$ with $x_j\in [0,q)$.
Let $T''$ be such that $T''\sim T_n=T'\sim T$ then $T''=\bigotimes_{j=1}^{n-1} [p-q+x_i,p+x_i] \times [2q+p,3q+p]$. Since $S$ is the unique small cube properly touching $T$ in $\mathbb{R}^{n-1}\times \{q\}$, $F:=([0,q]^{n-1}\setminus [0,p]^{n-1})\times [q,2q]$ is filled by big cubes properly touching $T$. Now consider $E:=(\bigotimes_{j=1}^{n-1} [p-q+x_i,p+x_i]\cap ([0,q]^{n-1}\setminus [0,p]^{n-1}))\times [2q,2q+p]$. $E$ is touching $F$ and $T''$ in full opposite sides of distance $p$ so it can only by filled by small cubes and since small cubes cannot touch, by at most one small cube (see Figure \ref{fig3}).

There is an $i\in [1,n-1]$ such that $x_i\geq p$ otherwise $T'$ and $P$ would properly intersect so w.l.o.g. assume $x_1\geq p$
\begin{claim}
$x_i=0$ for all $i\in[2,n-1]$.
\end{claim}
\begin{proofofclaim}
Assume w.l.o.g $x_2>0$ for $k\in [2,n-1]$. Then $x_1:=(p-3\epsilon,p+\epsilon,p-\epsilon,...,p-\epsilon,2q+\epsilon)^T$ and $x_2:=(p+\epsilon,p-3\epsilon,p-\epsilon,...,p-\epsilon,2q+\epsilon)^T$ are inner points of $E$ but their midpoint $x_m:=(x_1+x_2)/2=(p-\epsilon,...,p-\epsilon,2q+\epsilon)^T$ is an inner point of $T_n$ so $E$ cannot be filled by a single convex set.
\end{proofofclaim}

\begin{claim}
It holds that $x_1=p$.
\end{claim}
\begin{proofofclaim}
Assume $x_1>p$. Then $E':=[p,x_1]\times [0,p]^{n-2}\times [q,p+q]$ is touching $T'$ and $S$ in full opposite sides of distance at most $q-p$ so $E'$ cannot be filled.
\end{proofofclaim}

Therefore $T'=[p,p+q]\times[0,q]^{n-2}\times [q,2q]$ and since two different cubes of the form $\bigotimes_{j=1}^{i-1} [0,q]\times [p,p+q]\times\bigotimes_{j=i+1}^{n-1} [0,q]\times [q,2q]$ would intersect, $T'$ is the unique cube of this form.

\end{proof}
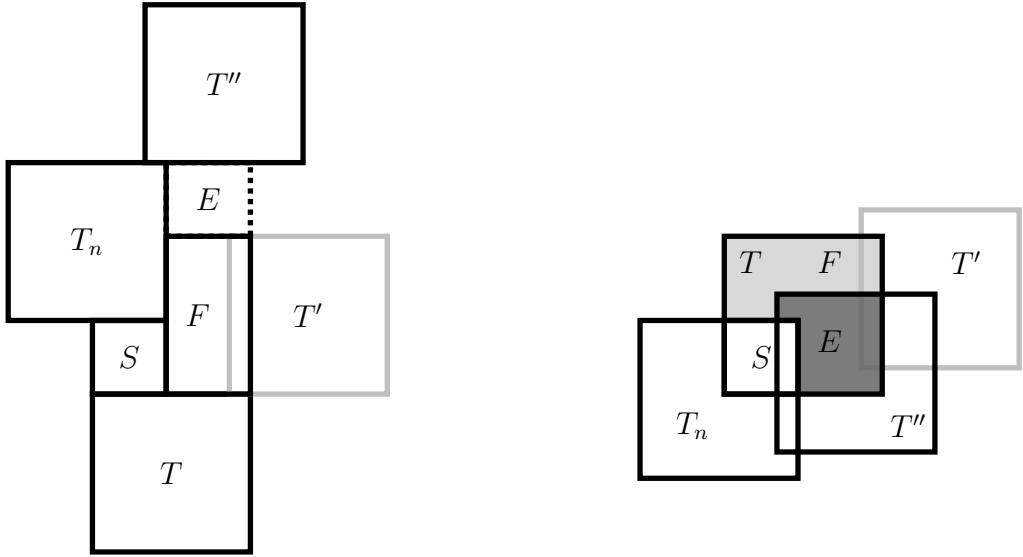
\begin{figure}[H]
\begin{tikzpicture}[scale=0.7]

\begin{scope}[xshift=0cm,yshift=0cm]
\draw[-,line width=2pt] (0,0) -- (0,3) -- (3,3) -- (3,0) -- cycle;
\node  at (1.5,1.5) {$T$};
\end{scope}

\begin{scope}[xshift=2.6cm,yshift=3cm]
\draw[-,color=lightgray, line width=2pt] (0,0) -- (0,3) -- (3,3) -- (3,0) -- cycle;
\node  at (1.5,1.5) {$T'$};
\end{scope}

\begin{scope}[xshift=1cm,yshift=7.4cm]
\draw[-, line width=2pt] (0,0) -- (0,3) -- (3,3) -- (3,0) -- cycle;
\node  at (1.5,1.5) {$T''$};
\end{scope}

\begin{scope}[xshift=-1.6cm,yshift=4.4cm]
\draw[-, line width=2pt] (0,0) -- (0,3) -- (3,3) -- (3,0) -- cycle;
\node  at (1.5,1.5) {$T_n$};
\end{scope}

\begin{scope}[xshift=0cm,yshift=3cm]
\draw[-, line width=2pt] (0,0) -- (0,1.4) -- (1.4,1.4) -- (1.4,0) -- cycle;
\node  at (0.7,0.7) {$S$};
\end{scope}

\draw[line width=2pt] (1.4,3) -- (1.4,6) -- (3,6) -- (3,3) -- cycle;
\node  at (2,4.5) {$F$};

\draw[dotted, line width=2pt] (1.4,7.4) -- (1.4,6) -- (3,6) -- (3,7.4) -- cycle;
\node  at (2.2,6.7) {$E$};

\begin{scope}[xshift=12cm,yshift=3cm]

\begin{scope}[xshift=2.6cm,yshift=0.5cm]
\draw[-,color=lightgray,line width=2pt] (0,0) -- (0,3) -- (3,3) -- (3,0) -- cycle;
\node  at (2,2) {$T'$};
\end{scope}

\fill[lightgray,opacity=0.6] (0,1.4) -- (0,3) -- (3,3) -- (3,0) -- (1.4,0) -- (1.4,1.4) -- cycle;
\node  at (2,2.5) {$F$};

\fill[darkgray,opacity=0.6] (1.4,0) -- (1.4,1.4) -- (1,1.4) -- (1,1.9) -- (3,1.9) -- (3,0) -- cycle;
\node  at (2,1) {$E$};

\draw[-, line width=2pt] (0,0) -- (0,3) -- (3,3) -- (3,0) -- cycle;
\node  at (0.5,2.5) {$T$};

\begin{scope}[xshift=0cm,yshift=0cm]
\draw[-, line width=2pt] (0,0) -- (0,1.4) -- (1.4,1.4) -- (1.4,0) -- cycle;
\node  at (0.7,0.7) {$S$};
\end{scope}

\begin{scope}[xshift=-1.6cm,yshift=-1.6cm]
\draw[-,line width=2pt] (0,0) -- (0,3) -- (3,3) -- (3,0) -- cycle;
\node  at (1,1) {$T_n$};
\end{scope}

\begin{scope}[xshift=1cm,yshift=-1.1cm]
\draw[-,line width=2pt] (0,0) -- (0,3) -- (3,3) -- (3,0) -- cycle;
\node  at (2.5,0.5) {$T''$};
\end{scope}

\end{scope}

\end{tikzpicture}
\caption{The 3-dimensional setup in the proof of Lemma \ref{uniqueneighbour} seen in cross section (left) and from above (right).}
\label{fig3}
\end{figure}

Now again let $S$ be a small cube in the tiling and w.l.o.g let $S:=[-p/2,p/2]^n$. Let $\mathcal{T}:=\{T_1^-,...,T_n^-,T_1^+,...,T_n^+\}$ be the set of cubes properly touching $S$ such that $S$ touches $T_i^-$ in $\mathbb{R}^{i-1}\times \{-p/2\}\times \mathbb{R}^{n-i}$ and $T_i^+$ in $\mathbb{R}^{i-1}\times \{p/2\}\times \mathbb{R}^{n-i}$ for all $i\in [1,n]$ and $\mathcal{C}:=\{c_1^-,...,c_n^-,c_1^+,...,c_n^+\}$ their corresponding center point.
We can now define a function $\nu:\mathcal{T}\rightarrow \mathcal{T}$ that assigns every cube $T\in\mathcal{T}$ the cube $T'$ such that $(S,T,T')$ are in the same relation as $(S,T,T_i)$ from Lemma \ref{uniqueneighbour} with the uniquely determined $i$. We also use $\nu$ as a function $\nu:\mathcal{C}\rightarrow \mathcal{C}$ with the obvious meaning. 
Consequently, for all $ i\in [1,n]$ there exists $j\in [1,n]\setminus\{i\}$ such that $\nu(v_i^-)-v_i^-=qe_i\pm pe_j$ and for symmetry reasons $\nu(v_i^+)-v_i^+=-qe_i\mp pe_j$.
Now let $G:=(V,E)$ be the directed graph with vertex set $V:=\mathcal{T}$ and edge set $E:=\{(T,\nu(T))|T\in\mathcal{T}\}$ (see Figure \ref{fig4}).

\begin{figure}[H]
\begin{tikzpicture}[yscale=1,scale=0.7]

\begin{scope}[xshift=1cm,yshift=0cm]

\draw[-,line width=2pt] (0,3) -- (0,5) -- (2,5) -- (2,3) -- cycle;
\end{scope}

\begin{scope}[xshift=0cm,yshift=0cm]

\draw[-,line width=2pt] (0,0) -- (0,3) -- (3,3) -- (3,0) -- cycle;
\end{scope}

\begin{scope}[xshift=3cm,yshift=2cm]

\draw[-, line width=2pt] (0,0) -- (0,3) -- (3,3) -- (3,0) -- cycle;
\end{scope}

\begin{scope}[xshift=-2cm,yshift=3cm]

\draw[-, line width=2pt] (0,0) -- (0,3) -- (3,3) -- (3,0) -- cycle;
\end{scope}

\begin{scope}[xshift=1cm,yshift=5cm]

\draw[-, line width=2pt] (0,0) -- (0,3) -- (3,3) -- (3,0) -- cycle;
\end{scope}

\fill (1.5,1.5) circle (0.2);
\fill (2.5,6.5) circle (0.2);
\fill (-0.5,4.5) circle (0.2);
\fill (4.5,3.5) circle (0.2);

\draw (1.5,0.75) node {$T_2^-$};
\draw (2.5,7.25) node {$T_2^+$};
\draw (-1.25,4.5) node {$T_1^-$};
\draw (5.25,3.5) node {$T_1^+$};
\draw (2,4) node {$S$};

\draw[->, line width=2pt] (1.5,1.5) -- (-0.44,4.41);
\draw[->, line width=2pt] (-0.5,4.5) -- (2.41,6.44);
\draw[->, line width=2pt] (2.5,6.5) -- (4.41,3.56);
\draw[->, line width=2pt] (4.5,3.5) --(1.56,1.59);

\draw (4,2.5) node {$\nu$};

\end{tikzpicture}
\caption{}
\label{fig4}
\end{figure}

\begin{lemma}
$G$ is a cycle of length $n$ and $\nu^n(T_i^\pm)=T_i^\mp$.
\end{lemma}

\begin{proof}
Every vertex in $G$ has out-degree one so $G$ contains a cycle. Let $C:=(\bar{T}_1,...,\bar{T}_k)$  be a cycle in $G$ with center points $(\bar{c}_1,...,\bar{c}_k)$, $\bar{T}_1=\bar{T}_k$   and $3\leq k\leq 2n+1$.
For $i\in [1,k-1]$ let $j_i\in [1,n]$ such that $\bar{T}_i\in\{T_{j_i}^-,T_{j_i}^+\}$. Then $\bar{c}_i$ can be written as $$\bar{c}_i=x_i\frac{q+p}{2}e_{j_i}+x_{i+1}\frac{q-p}{2}e_{j_{i+1}}+\sum_{l\in S_i^-}\frac{p-q}{2}e_l+\sum_{l\in S_i^+}\frac{q-p}{2}e_l$$ where $\{j_i,j_{i+1}\}\uplus S_i^- \uplus S_i^+=[1,n]$ and $x_i,x_{i+1}\in \{-1,1\}$ and $$\bar{c}_{i+1}=-x_i\frac{q-p}{2}e_{j_i}+x_{i+1}\frac{q+p}{2}e_{j_{i+1}}+\sum_{l\in S_i^-}\frac{p-q}{2}e_l+\sum_{l\in S_i^+}\frac{q-p}{2}e_l.$$
Let $v_i:=\bar{c}_{i+1}-\bar{c}_i=-x_iqe_{j_i}+x_{i+1}pe_{j_{i+1}}$, then $\sum_{i=1}^{k-1} v_i=0$ because $c_k=c_1$.

$$0=\sum_{i=1}^{k-1} v_i=\sum_{i=1}^{k-1}-x_iqe_{j_i}+x_{i+1}pe_{j_{i+1}}=\sum_{i=1}^{k-1}x_i(p-q)e_{j_i}.$$

Therefore $T_{j}^-\in C \iff T_{j}^+\in C$ for all $j\in[1,n]$ and $C$ is an even cycle. W.l.o.g. let $\bar{T}_1=T_{1}^-$, then for symmetry reasons  $\bar{T}_{k/2+1}=T_{1}^+$ and 

$$(p+q)e_1+\sum_{l=2}^{n}\pm (q-p)e_l=v_{k/2+1}-v_1=\sum_{i=1}^{k/2} v_i=\sum_{i=1}^{k/2}-x_iqe_{j_i}+x_{i+1}pe_{j_{i+1}}$$

where $e_{j_1}=e_{j_{k/2+1}}=e_1$ and $\{e_{j_1},...,e_{j_{k/2}}\}$ are pairwise different. Consequently, $k/2=n$ which concludes the proof.

\end{proof}

Now consider the sets of vectors $\{b_1,b_2,...,b_n\}\subseteq B\mathbb{Z}^n$ where $$b_i=\bar{c}_{i+1}-\bar{c}_i=\begin{cases}
pe_j+qe_l &\bar{T}_{i+1}=T_j^+ \land \bar{T}_{i}=T_l^-\\
-pe_j+qe_l &\bar{T}_{i+1}=T_j^- \land \bar{T}_{i}=T_l^-\\
pe_j-qe_l &\bar{T}_{i+1}=T_j^- \land \bar{T}_{i}=T_l^+\\
-pe_j-qe_l &\bar{T}_{i+1}=T_j^- \land \bar{T}_{i}=T_l^+.\\
\end{cases}$$
Let $C:=(b_1,b_2,...,b_n)$ then $SC=A$ where $S=(s_{i,j})_{i,j\in[1,n]}\in B'_n$ and 
$$s_{i,j}=\begin{cases}
1 &\bar{T}_{i}=T_j^- \\
-1 &\bar{T}_{i}=T_j^+\\
0 &else.\\

\end{cases}$$

Now $\det(C)=\det(A)=\det(B)=p^n+q^n$ and therefore $C\mathbb{Z}^n=B\mathbb{Z}^n$ and $A$ and $B$ describe the same tiling up to symmetries.

\subsection*{Acknowledgements}
 The author acknowledges the support of the Austrian Science Fund (FWF): W1230. The author also thanks Christian Elsholtz,  Mihalis Kolountzakis and the anonymous referee for comments on this manuscript.

\printbibliography

@article {MR1874434,
    AUTHOR = {B\"{o}lcskei, Attila},
     TITLE = {Filling space with cubes of two sizes},
   JOURNAL = {Publ. Math. Debrecen},
  FJOURNAL = {Publicationes Mathematicae Debrecen},
    VOLUME = {59},
      YEAR = {2001},
    NUMBER = {3-4},
     PAGES = {317--326},
      ISSN = {0033-3883},
   MRCLASS = {52C22},
  MRNUMBER = {1874434},
MRREVIEWER = {Alexey R. Alimov},
}

@article{hajos1942einfache,
  title={{\"U}ber einfache und mehrfache Bedeckung des $n$-dimensionalen Raumes mit einem W{\"u}rfelgitter},
  author={Haj{\'o}s, Georg},
  journal={Mathematische Zeitschrift},
  volume={47},
  number={1},
  pages={427--467},
  year={1942},
  publisher={Springer}
}

@article{keller1930uber,
  title={{\"U}ber die l{\"u}ckenlose Einf{\"u}llung des Raumes mit W{\"u}rfeln},
  author={Keller, Ott-Heinrich},
  journal={J. reine angew. Math},
  volume={163},
  pages={231--248},
  year={1930}
}

@book{minkowski1907diophantische,
  title={Diophantische Approximationen: Eine Einf{\"u}hrung in die Zahlentheorie},
  author={Minkowski, Hermann},
  volume={2},
  year={1907},
  publisher={BG Teubner}
}

@article{perron1940luckenlose,
  title={{\"U}ber l{\"u}ckenlose Ausf{\"u}llung des $n$-dimensionalen Raumes durch kongruente W{\"u}rfel},
  author={Perron, Oskar},
  journal={Mathematische Zeitschrift},
  volume={46},
  number={1},
  pages={1--26},
  year={1940},
  publisher={Springer}
}

@article{lagarias1992keller,
  title={Keller’s cube-tiling conjecture is false in high dimensions},
  author={Lagarias, Jeffrey C. and Shor, Peter W.},
  journal={Bulletin of the American Mathematical Society},
  volume={27},
  number={2},
  pages={279--283},
  year={1992}
}

@article{mackey2002cube,
  title={A cube tiling of dimension eight with no facesharing},
  author={Mackey, John},
  journal={Discrete and Computational Geometry},
  volume={28},
  number={2},
  pages={275--279},
  year={2002},
  publisher={Springer}
}

@inproceedings{brakensiek2020resolution,
  title={The resolution of Keller’s conjecture},
  author={Brakensiek, Joshua and Heule, Marijn and Mackey, John and Narv{\'a}ez, David},
  booktitle={International Joint Conference on Automated Reasoning},
  pages={48--65},
  year={2020},
  organization={Springer}
}

@book {grunbaum1987tilings,
    AUTHOR = {Gr\"{u}nbaum, Branko and Shephard, G. C.},
     TITLE = {Tilings and patterns},
    SERIES = {A Series of Books in the Mathematical Sciences},
      NOTE = {An introduction},
 PUBLISHER = {W. H. Freeman and Company, New York},
      YEAR = {1989},
     PAGES = {xii+446},
      ISBN = {0-7167-1998-3},
   MRCLASS = {52A45},
  MRNUMBER = {992195},
}

@article{martini,
	author = {Martini, Horst and Makai, Endre and Soltan, Valeriu},
	journal = {Beitrage zur Algebra und Geometrie},
	month = {04},
	pages = {481–495},
	title = {Unilateral tilings of the plane with squares of three sizes, Beiträge Algebra Geom. 39 (1998), no. 2, 481–495.},
	volume = {39},
	year = {1998}
}

@article{bolcskei2000classification,
  title={Classification of unilateral and equitransitive tilings by squares of three sizes},
  author={B{\"o}lcskei, Attila},
  journal={Beitr{\"a}ge zur Algebra und Geometrie},
  volume={41},
  number={1},
  pages={267--277},
  year={2000}
}

@article{schattschneider2000unilateral,
  title={Unilateral and equitransitive tilings by squares},
  author={Schattschneider, Doris},
  journal={Discrete \& Computational Geometry},
  volume={24},
  number={2},
  pages={519--526},
  year={2000},
  publisher={Springer}
}

@article{dawson1984filling,
	author = {Dawson, R.J.M.},
	doi = {10.1016/0097-3165(84)90007-4},
	issn = {0097-3165},
	journal = {Journal of Combinatorial Theory, Series A},
	number = {2},
	pages = {221–229},
	title = {On filling space with different integer cubes},
	url = {https://www.sciencedirect.com/science/article/pii/0097316584900074},
	volume = {36},
	year = {1984}
}

@inproceedings {baumert1971combinatorial,
    AUTHOR = {Baumert, L. D. and McEliece, R. J. and Rodemich, Eugene and
              Rumsey, Jr., Howard C. and Stanley, Richard and Taylor,
              Herbert},
     TITLE = {A combinatorial packing problem},
 BOOKTITLE = {Computers in algebra and number theory ({P}roc. {SIAM}-{AMS}
              {S}ympos. {A}ppl. {M}ath., {N}ew {Y}ork, 1970)},
     PAGES = {97--108. SIAM-AMS Proc., Vol. IV},
      YEAR = {1971},
   MRCLASS = {05B40},
  MRNUMBER = {0337668},
MRREVIEWER = {W. Moser},
}

@article{mathew2017new,
  title={New lower bounds for the Shannon capacity of odd cycles},
  author={Mathew, K.A. and {\"O}sterg{\aa}rd, P.R.J.},
  journal={Designs, Codes and Cryptography},
  volume={84},
  number={1},
  pages={13--22},
  year={2017},
  publisher={Springer}
}

@ARTICLE{Lovasz,
  author={Lovasz, L.},
  journal={IEEE Transactions on Information Theory}, 
  title={On the Shannon capacity of a graph}, 
  year={1979},
  volume={25},
  number={1},
  pages={1-7},
  doi={10.1109/TIT.1979.1055985}}

@ARTICLE{Shannon,
  author={Shannon, C.},
  journal={IRE Transactions on Information Theory}, 
  title={The zero error capacity of a noisy channel}, 
  year={1956},
  volume={2},
  number={3},
  pages={8-19},
  doi={10.1109/TIT.1956.1056798}}

@article{sprague1939beispiel,
  title={Beispiel einer Zerlegung des Quadrats in lauter verschiedene Quadrate},
  author={Sprague, Roland},
  journal={Mathematische Zeitschrift},
  volume={45},
  number={1},
  pages={607--608},
  year={1939},
  publisher={Springer}
}

@article{MEIR1968126,
	author = {A. Meir and L. Moser},
	doi = {10.1016/S0021-9800(68)80047-X},
	issn = {0021-9800},
	journal = {Journal of Combinatorial Theory},
	number = {2},
	pages = {126–134},
	title = {On packing of squares and cubes},
	url = {https://www.sciencedirect.com/science/article/pii/S002198006880047X},
	volume = {5},
	year = {1968}
}

@article{januszewski2021note,
  title={A note on perfect packing of squares and cubes},
  author={Januszewski, J. and Zielonka, {\L}.},
  journal={Acta Mathematica Hungarica},
  volume={163},
  number={2},
  pages={530--537},
  year={2021},
  publisher={Springer}
}

@misc{tao2022perfectly,
      title={Perfectly packing a square by squares of nearly harmonic sidelength}, 
      author={Terence Tao},
      year={2022},
      eprint={2202.03594},
      archivePrefix={arXiv},
      primaryClass={math.MG}
}

@article{zong2005known,
  title={What is known about unit cubes},
  author={Zong, Chuanming},
  journal={Bulletin of the American Mathematical Society},
  volume={42},
  number={2},
  pages={181--211},
  year={2005}
}

@article {MR511994,
    AUTHOR = {Duijvestijn, A. J. W.},
     TITLE = {Simple perfect squared square of lowest order},
   JOURNAL = {J. Combin. Theory Ser. B},
  FJOURNAL = {Journal of Combinatorial Theory. Series B},
    VOLUME = {25},
      YEAR = {1978},
    NUMBER = {2},
     PAGES = {240--243},
      ISSN = {0095-8956},
   MRCLASS = {05B30},
  MRNUMBER = {511994},
MRREVIEWER = {W. T. Tutte},
       DOI = {10.1016/0095-8956(78)90041-2},
       URL = {https://doi.org/10.1016/0095-8956(78)90041-2},
}

@article {MR1605065,
    AUTHOR = {Kolountzakis, Mihail},
     TITLE = {Lattice tilings by cubes: whole, notched and extended},
   JOURNAL = {Electron. J. Combin.},
  FJOURNAL = {Electronic Journal of Combinatorics},
    VOLUME = {5},
      YEAR = {1998},
     PAGES = {Research Paper 14, 11},
   MRCLASS = {52C22 (11H31)},
  MRNUMBER = {1605065},
MRREVIEWER = {Richard Kenyon},
       URL = {http://www.combinatorics.org/Volume_5/Abstracts/v5i1r14.html},
}

\end{document}